\documentclass[a4paper]{article}

\usepackage[utf8]{inputenc}
\usepackage{lmodern}
\usepackage{amssymb,amsfonts,amsthm,amsmath,bbm,mathrsfs,aliascnt,mathtools}
\usepackage[english]{babel}
\usepackage[colorlinks]{hyperref}
\usepackage{tikz}
\usetikzlibrary{decorations}
\usetikzlibrary{decorations.pathreplacing}
\tikzset{individu/.style={draw,thick}}

\theoremstyle{plain}
\newtheorem{theorem}{Theorem}[section]

\newtheorem{lemma}[theorem]{Lemma}
\newtheorem{proposition}[theorem]{Proposition}

\theoremstyle{definition}

\theoremstyle{remark}
\newtheorem{remark}[theorem]{Remark}

\numberwithin{equation}{section}

\usepackage{tikz}

\makeatletter \newcommand \listoftodos{\section*{Todo list} \@starttoc{tdo}}
\newcommand\l@todo[2]
{\par\noindent \textit{#2}, \parbox{10cm}{#1}\par} \makeatother

\newcommand{\N}{\mathbb{N}}

\newcommand{\R}{\mathbb{R}}
\newcommand{\C}{\mathbb{C}}

\newcommand{\calN}{\mathcal{N}}

\newcommand{\ind}[1]{\mathbf{1}_{\left\{#1\right\}}}

\newcommand{\crochet}[1]{{\langle #1 \rangle}}

\renewcommand{\hat}[1]{\widehat{#1}}

\newcommand{\e}{\mathrm{e}}
\newcommand{\dd}{\mathrm{d}}
\newcommand{\egaldistr}{{\overset{(d)}{=}}}

\DeclareMathOperator{\E}{\mathbb{E}}

\renewcommand{\P}{\mathbb{P}}

\newcommand{\calF}{\mathcal{F}}

\newcommand{\calP}{\mathcal{P}}

\newcommand{\x}{\mathbf{x}}

\renewcommand{\rho}{\varrho}
\renewcommand{\epsilon}{\varepsilon}

\title{Biggins' Martingale Convergence for~Branching~L\'evy~Processes}
\author{Jean Bertoin\thanks{Institute of Mathematics, University of Zurich, Switzerland.} \and Bastien Mallein\thanks{LAGA - Institut Galil\'ee, Universit\'e Paris 13, France.}}
\date{\today}

\begin{document}

\maketitle

\begin{abstract}
A branching Lévy process can be seen as the continuous-time version of a branching random walk.  It describes a particle system on the real line in which particles move and reproduce independently in a Poissonian manner. Just as for L\'evy processes, the law of a branching Lévy process is determined by its characteristic triplet $(\sigma^2,a,\Lambda)$, where the branching L\'evy measure $\Lambda$ describes the intensity of the Poisson point process of births and jumps. We establish a version of Biggins' theorem in this framework, that is we provide necessary and sufficient conditions  in terms of the characteristic triplet $(\sigma^2,a,\Lambda)$  for additive martingales  to have a non-degenerate limit.
\end{abstract}


\medskip


\section{Introduction and main result}
\label{sec:introduction}

We start by introducing some notation. We denote by $\x=(x_n)_{ n\geq 1}$ a generic non-increasing sequence in $[-\infty,\infty)$ with $\lim_{n\to \infty}x_n=-\infty$. We view $\x$ as a ranked sequence of positions of particles in $\R$, with the convention that possible particles located at $-\infty$ should be thought of as non-existing (so particles never accumulate in $\R$ and the number of particles may be finite or infinite). We thus often identify $\x$ with a locally finite point measure on $\mathbb{R}$, $\sum \delta_{x_n}$, where, by convention, the possible atoms at $-\infty$ are discarded in this sum. We write $\calP$ for the space of such sequences or point measures. 

Then let $(Z_n)_{n\geq 0}$ be a branching random walk with reproduction law $\pi$, where $\pi$ is some probability measure on $\calP$. In words, this process starts at generation $0$ with a single particle at $0$ and the law of $Z_1$ is given by $\pi$. For every particle at generation $n\geq 1$, say at position $x\in\R$, the sequence of positions of the children of that particle is given by $x+Y$, where $Y$ has the law $\pi$, and to different particles correspond independent copies of $Y$ with law $\pi$. 

A classical assumption made to ensure the well-definition of $(Z_n)$, i.e. that for all $n \in \N$ there are only finitely many particles in the positive half-line, is that there exists $\theta \geq 0$ such that
\begin{equation}
  \label{eqn:integrabilityBRW}
 m(\theta)\coloneqq \int_{\calP} \crochet {\x, \e_{\theta}}\pi(\dd \x) = \E\left( \crochet{Z_1,\e_\theta}\right)<\infty
\end{equation}
where we denote by $\crochet{\x,g} = \sum_{n \geq 1} g(x_n)$ for all measurable nonnegative functions $g$, and $\e_\theta : x \in \R \mapsto \e^{\theta x}$. In particular, we have $\crochet {\x, \e_{\theta}}=\sum \e^{\theta x_n}$. It is common knowledge --and a simple application of the branching property-- that $\E(\crochet{Z_n,\e_\theta}) = m(\theta)^n$ and that the process
$$W_n\coloneqq m(\theta)^{-n}\crochet {Z_n, \e_{\theta}}, \qquad n\geq 0$$
is a nonnegative martingale. The question of whether its terminal value $W_{\infty}$ is non-degenerate has a fundamental importance and was solved by Biggins \cite{Big77} under the additional assumption that
\begin{equation}
  \label{eqn:derivativeBRW}
  m'(\theta)\coloneqq \int_{\calP}\sum x_j\e^{\theta x_j} \pi(\dd \x) \quad \text{ exists and is finite.}
\end{equation}
Note that by \eqref{eqn:integrabilityBRW}, $m$ can be defined, for any $z \in \C$ with $\Re z =\theta$ by
\[
  m(z) \coloneqq \int_{\calP} \crochet{\x, \e_z} \pi(\dd \x) = \E\left( \crochet{Z_1,\e_z} \right),
\]
in which case $m'(\theta)$ is the complex derivative of the function $m$ at point $\theta$, justifying the notation in \eqref{eqn:derivativeBRW}.

Specifically, \cite[Lemma 5]{Big77} states that $\E(W_\infty) = 1$, or equivalently that $(W_n)_{n \geq 0}$ is uniformly integrable, if and only if 
\begin{equation}
  \label{condbigg} \theta m'(\theta)/m(\theta) < \log m(\theta) \ \text{ and }\ \int_{\calP} \crochet {\x, \e_{\theta}}\log^+ \crochet {\x, \e_{\theta}}\,  \pi(\dd \x)<\infty.
\end{equation}
If \eqref{condbigg} does not hold, then $W_{\infty}=0$ a.s. This result has later been improved by Alsmeyer and Iksanov \cite{AlI09}, who obtained a necessary and sufficient condition for the uniform integrability of $(W_n)_{n\geq 0}$ without the additional integrability condition \eqref{eqn:derivativeBRW}.

Recall that, by log-convexity of the function $m$, the first inequality of \eqref{condbigg} entails that $m(0)=\E(\crochet{Z_1,1})>1$, i.e. the Galton-Watson process $(\crochet{Z_n,1})_{n \geq 0}$ is supercritical. In particular, the branching random walk $Z$ survives with positive probability. Biggins \cite{Big77} further pointed out that when the martingale $(W_n)_{n\geq 0}$ is uniformly integrable, the event $\{W_{\infty}>0\}$ actually coincides a.s. with the non-extinction event of the branching random walk.

The purpose of this work is to present a version of Biggins' martingale convergence theorem for branching L\'evy processes, a family of branching processes in continuous time that was recently introduced in \cite{BeM17+}. Branching L\'evy processes bear the same relation to branching random walks as L\'evy processes do to random walks: a branching Lévy process $(Z_t)_{t\geq 0}$ is a point-measure valued process such that for every $r>0$, its discrete-time skeleton $(Z_{nr})_{n\geq 0}$ is a branching random walk. This is a natural extension of the notion of continuous-time branching random walks\footnote{Which can be thought of as branching compound Poisson processes.} as considered by Uchiyama \cite{Uch82}, or the family of branching Lévy processes considered by Kyprianou \cite{Kyp1999}; another subclass also appeared in the framework of so-called compensated-fragmentation processes, see \cite{BeCF}. 

The dynamics of a branching L\'evy process can be described informally as follows. The process starts at time $0$ with a unique particle located at the origin. As time passes, this particle moves according to a certain Lévy process, while making children around its position in a Poissonian fashion. Each of the newborn particles immediately starts an independent copy of this branching Lévy process from its current position. We stress that  a jump of a particle may be correlated with its offspring born at the same time.

The law of a branching Lévy process $(Z_t)_{t\geq 0}$ is characterized by a triplet $(\sigma^2,a,\Lambda)$, where $\sigma^2 \geq 0$, $a \in \R$ and $\Lambda$ is a sigma-finite measure on $\calP$ without atom at $\{(0,-\infty,-\infty,...)\}$, which satisfies
\begin{equation}
  \label{eqn:lambdaLevyProcess}
  \int_\mathcal{P} (1 \wedge x_1^2) \Lambda(\dd \x) < \infty.
\end{equation}
Furthermore, we need another integrability condition for $\Lambda$ that depends on a parameter $\theta\geq 0$;
which is henceforth fixed once for all. Specifically, we request 
\begin{equation}
  \label{eqn:exponentialIntegrability}
  \int_{\calP}\big({\mathbf 1}_{\{x_1>1\}} \e^{\theta x_1}+\sum_{k\geq 2} \e^{\theta x_k} 
  \big)\Lambda(\dd\x) < \infty.
\end{equation}
The term $\sigma^2$ is the Brownian variance coefficient of the trajectory of a particle, $a$ is the drift term, and the branching L\'evy measure $\Lambda$ encodes both the distribution of the jumps of particles, and the branching rate and distribution of their children. The assumption \eqref{eqn:exponentialIntegrability} guarantees the well-definition and the absence of local explosion in the branching Lévy process.

The integrability conditions \eqref{eqn:lambdaLevyProcess} and \eqref{eqn:exponentialIntegrability} enable us to define 
for every $z\in\C$ with $\Re z=\theta$ 
\begin{equation}
  \label{eqn:defineCumulant}
  \kappa(z)\coloneqq \frac{1}{2}\sigma^2z^2 + a z + \int_{\calP} \left( \e^{z x_1} -1-z x_1\ind{|x_1|<1}+ \sum_{k\geq 2} \e^{z x_k}\right) \Lambda(\dd \x).
\end{equation}
We call $\kappa$ the cumulant generating function of $Z_1$; to justify the terminology, recall from Theorem 1.1(ii) in \cite{BeM17+}  that for all $t \geq 0$, we have
\[
  \E\left( \crochet{Z_t, \e_z} \right) = \exp\left( t \kappa(z) \right).
\]
In particular, in terms of the (skeleton) branching random walk $(Z_n)_{n\geq0}$ obtained by sampling $Z$ at integer times, we have the identities
$$m(\theta)=\exp(\kappa(\theta))\ \ \text{and}\ \ m'(\theta)= \kappa'(\theta)\exp(\kappa(\theta)),$$
where $\pi$ is the law of $Z_1$, $m(\theta)$ and $m'(\theta)$ are defined in \eqref{eqn:integrabilityBRW} and \eqref{eqn:derivativeBRW}, and
\begin{equation}\label{kappa'}
\kappa'(\theta) = \sigma^2\theta + a + \int_{\calP} \left(x_1(\e^{\theta x_1}-\ind{|x_1|<1})+ \sum_{k\geq 2}^{\infty} x_k\e^{\theta x_k}\right) \Lambda(\dd \x).
\end{equation}
The well-definition and finiteness of the above integral is equivalent to the well-definition and finiteness of $m'(\theta)$.
Throughout the rest of the article, we assume $\kappa'(\theta)$ in \eqref{kappa'} to be well-defined and finite.

We are now able to state our version of Biggins' martingale convergence theorem in branching Lévy processes settings.
\begin{theorem}
\label{T1}
Let $(Z_t)_{t\geq 0}$ be a branching L\'evy process with characteristic triplet $(\sigma^2,a,\Lambda)$. The martingale $W$ given by
$$W_t\coloneqq \exp(-t\kappa(\theta))\crochet {Z_t, \e_{\theta}} \quad \text{for all }t\geq 0,$$
is uniformly integrable if and only if
\begin{equation}
  \label{cond1}
  \theta \kappa'(\theta)<\kappa(\theta)
\end{equation}
and
\begin{equation}
  \label{cond2}
  \int_{\cal P} \crochet {\x, \e_{\theta}} \left(\log \crochet {\x, \e_{\theta}}-1\right)^+ \Lambda(\dd \x)<\infty.
\end{equation}
Otherwise, the terminal value $W_{\infty}$ equals $0$ a.s.
\end{theorem}

\begin{remark} When the branching L\'evy measure $\Lambda$ is finite, the integrability condition \eqref{cond2}
is equivalent to the analog of  \eqref{condbigg}, namely 
$$  \int_{\cal P} \crochet {\x, \e_{\theta}}\log^+ \crochet {\x, \e_{\theta}} \Lambda(\dd \x)<\infty.$$
However, when $\Lambda$ is an infinite measure, the inequality above is a strictly stronger requirement than \eqref{cond2}.
\end{remark} 

Of course, the continuous time martingale $W$  is uniformly integrable if and only if this is the case for its discrete time skeleton $(W_n)_{n\geq 0}$, and one might expect that our statement should readily be reduced to Biggins' theorem. Condition \eqref{cond1} should certainly not come as a surprise, since it merely rephrases the first inequality in \eqref{condbigg}. Thus everything boils down to verifying that Condition \eqref{cond2} is equivalent to the $L\log^+ L$ integrability condition in \eqref{condbigg}.

However, the latter does not seems to have a straightforward proof (at least when $\Lambda$ is infinite), the difficulty stems from the fact that there is no simple expression for the law $\pi$ of $Z_1$ in terms of the characteristics $(\sigma^2,a,\Lambda)$. Specifically, we cannot evaluate directly $\E(\crochet{Z_1,\e_{\theta} }\log^+ \crochet{Z_1,\e_{\theta} })$; only expectations of linear functionals of $Z_1$ can be computed explicitly in terms of the characteristics of the branching L\'evy process. 
We shall thus rather establish Theorem~\ref{T1} by an adaptation of the arguments of Lyons \cite{Lyo97} for proving Biggins' martingale convergence for branching random walks, using a version of the celebrated spinal decomposition,
and properties of Poisson random measures.

\begin{remark}
It is well-known that for branching random walks, the law of the terminal value $W_\infty$ is a fix point of a smoothing transform (see e.g. Liu \cite{Liu}), more precisely
\begin{equation}
  \label{eqn:smoothing}
  W_\infty \egaldistr \sum_{j \in \N} \e^{\theta x_j - t \kappa(\theta)} W_\infty^{(j)},
\end{equation}
where $\x = (x_n)$ is a random variable in $\mathcal{P}$ with same law as $Z_1$, and $(W_\infty^{(j)})$ are i.i.d. copies of $W_\infty$ independent of $\x$. As observed above, the law of $Z_1$ cannot be obtained as a simple expression in terms of the characteristic of a branching Lévy process. However,  using classical approximation techniques, one can still get a functional equation for the Laplace transform of $W_\infty$. More precisely, setting $w(y) = \E\left( \exp\left( \e^{-\theta y} W_\infty\right)  \right)$, \eqref{eqn:smoothing} yields
\[
  \forall y \in \R, \quad w(y) = \E\left( \prod_{j \in \N} w(y - x_j + t c_\theta) \right),
\]
with $\x$ sampled again with same law as $Z_1$ and $c_\theta = \kappa(\theta)/\theta$. Using approximation by branching Lévy processes with finite birth intensity, one can then check that $w$ is a solution of the equation
\begin{equation*}
  \frac{1}{2} \sigma^2 w''(y) + (c_\theta - a) w'(y) + \int_{\mathcal{P}} \prod_{j \in \N} w(y-x_j) - w(y) + x_1 \ind{|x_1| < 1} w'(y) \Lambda( \dd \x) = 0,
\end{equation*}
i.e. a traveling wave solution of a generalized growth-fragmentation equation. We refer to Berestycki, Harris  and Kyprianou \cite{BeHaKy} for a detailed study in the framework of homogeneous fragmentations. In particular, observe that the law of $W_\infty$ does not depend on the value of characteristic $a$ of the branching Lévy process.
\end{remark}

In the same vein, recall from Theorem 1 of Biggins \cite{biggins1992} that for $p\in(1,2]$, the martingale $W$ converges in $p$-th mean  whenever 
$$\E(W_1^p)<\infty \quad \text{and} \quad \kappa(p\theta)< p\kappa(\theta).$$
The same approach also enables us to make this criterion explicit in terms of the branching L\'evy measure $\Lambda$.
\begin{proposition}\label{P1} Let $p\in(1,2]$. If
$\kappa(p\theta)< p\kappa(\theta)$,  
\begin{equation}\label{cond3}
\int_{\calP} \crochet{ \x, \e_{\theta}}^p\ind{\crochet{ \x, \e_{\theta}}>2} \Lambda(\dd \x)<\infty,
\end{equation}
and $\kappa(q\theta)<\infty$ for some $q>p$, 
then the martingale $W$ is bounded in $L^p$. 
\end{proposition}

\begin{remark} When the branching L\'evy measure $\Lambda$ is finite,  \eqref{cond3}
is equivalent to the simpler
$\int_{\cal P} \crochet{ \x, \e_{\theta}}^p \Lambda(\dd \x)<\infty$. However, when $\Lambda$ is infinite, 
one always has that\footnote{Indeed, for all $\epsilon>0$, \eqref{eqn:lambdaLevyProcess} implies that $\Lambda\left(|x_1|>\epsilon \right) <\infty$ and \eqref{eqn:exponentialIntegrability} that $\Lambda\left(\sum_{j=2}^\infty \e^{\theta x_j} > \epsilon\right) < \infty$, thus $\Lambda\left( \crochet{\x,\e_\theta} \not \in [1-\delta,1+\delta]\right) < \infty$ for all $\delta > 0$.}  $\Lambda(1/2\leq \crochet{ \x, \e_{\theta}}\leq 2)=\infty$, which explains the role of the indicator function in~\eqref{cond3}. The additional assumption that $\kappa(q\theta)<\infty$ for some $q>p$ is also needed in our proof
to bound the contribution of the infinitely many birth events with $\crochet{ \x, \e_{\theta}}\leq 2$.
\end{remark} 

We do not address here the issue of uniform convergence in the variable~$\theta$; see Biggins \cite{biggins1992} for branching random walks, and further Theorem 2.3 in Dadoun \cite{Dad} in the setting of compensated fragmentations. However, as observed in \cite{biggins1992}, Proposition~\ref{P1} is a key step in this direction.
 
The two statements of this Introduction are established in the next section. 

\section{Proofs}
In this section, we start by summarizing the construction of the branching L\'evy process with characteristics $(\sigma^2, a, \Lambda)$ as a particle system, referring to Sections 4 and 5 in \cite{BeM17+} for a detailed account. We shall then present a version of the spinal decomposition tailored for the purpose of this proof, and finally adapt the approach of Lyons \cite{Lyo97} to establish Theorem~\ref{T1} and Proposition~\ref{P1}.

We first consider a Poisson point process $\calN (\dd t, \dd \x)$ on $[0,\infty)\times \calP$ with intensity $\dd t \otimes \Lambda(\dd \x)$, and an independent Brownian motion $(B_t)_{t\geq 0}$. Thanks to the assumptions \eqref{eqn:lambdaLevyProcess} and \eqref{eqn:exponentialIntegrability}, we can define
$$\xi_t\coloneqq \sigma B_t + at+ \int_{[0,t]\times \calP} x_1\ind{|x_1|<1} \calN^{\rm(c)}(\dd s,\dd \x) + \int_{[0,t]\times \calP} x_1\ind{|x_1|\geq 1} \calN(\dd s,\dd \x)$$
for every $t\geq 0$, where the first Poissonian integral is taken in the compensated sense; see e.g. Section 12.1 in Last and Penrose \cite{LastPen}. So $(\xi_t)_{t\geq 0}$ is a L\'evy process with characteristic exponent $\Phi$ given by the L\'evy-Khintchin formula
$$\Phi(r)\coloneqq -\frac{\sigma^2}{2} r^2 + i a r+\int_{\calP}\left(\e^{irx_1}-1-irx_1\ind{|x_1|<1}\right)\Lambda(\dd \x), \quad r\in\R,$$
in the sense that $\E(\exp(ir\xi_t))=\exp(t\Phi(r))$. 

One should view $(\xi_t)_{t\geq 0}$ as describing the trajectory of the initial particle in the process (the Eve particle in the terminology of \cite{BertoinMarkov}). Further, for each atom of $\calN$, say $(t,\x)$, we view $t$ as the time at which  the Eve particle jumps from position $\xi_{t-}$ to $\xi_t = \xi_{t-} + x_1$, while begetting a sequence of children located at $\xi_{t-}+x_2, \xi_{t-}+x_3, \ldots$. Then, using independent copies of $(\calN,B)$, we let in turn each newborn particle evolve (starting from its own birth time and location) and give birth to its own progeny just as the Eve particle, and so on, and so forth. The branching L\'evy process $Z=(Z_t)_{t\geq 0}$ is then obtained by letting $Z_t$ denote the random point measure whose atoms are given by the positions of the particles in the system at time $t$. 

We then introduce the tilted branching L\'evy measure $\hat{\Lambda}$ on $\calP$, defined by
$$  \hat{\Lambda}(\dd\x) \coloneqq \crochet{\x,\e_{\theta}} \Lambda(\dd \x),$$
and point first at the following elementary fact:
\begin{lemma}
\label{L3} If \eqref{cond2} is fulfilled, then it holds for every $c>0$ that
$$\int_{0}^{\infty} \hat \Lambda(\crochet{\x,\e_{\theta}}>\e^{ct}+1)\dd t <\infty;$$
whereas if \eqref{cond2} fails, then it holds for every $c>0$ and $s>0$ that
$$\int_{s}^{\infty} \hat \Lambda(\crochet{\x,\e_{\theta}}>\e^{ct})\dd t =\infty.$$
\end{lemma}

\begin{proof} Note first the identities 
\begin{align*}
  \int_{0}^{\infty} \hat \Lambda(\crochet{\x,\e_{\theta}}>\e^t+1)\dd t
  &= \int_{0}^{\infty} \dd t  \int_{\cal P} \Lambda(\dd \x) \crochet{\x,\e_{\theta}} \ind{\crochet{\x,\e_{\theta}}>\e^t+1}\\
  &= \int_{\cal P} \crochet {\x, \e_{\theta}} \left(\log  \crochet {\x, \e_{\theta}}-1\right)^+  \Lambda(\dd \x).
\end{align*}
Since \eqref{eqn:lambdaLevyProcess} and \eqref{eqn:exponentialIntegrability} readily entail $\hat \Lambda(\crochet{\x,\e_{\theta}}>b)<\infty$ for every $b>1$, the first claim follows. The proof for the second is similar.
\end{proof}

We next prepare some material for the spinal decomposition. We write $\P$ for the law of $(Z_t)_{t\geq 0}$, $(\calF_t)_{t\geq 0}$ for its natural filtration, and use the martingale $W=(W_t)_{t\geq 0}$ to introduce the tilted probability measure
\[
  \hat{\P}_{\mid \calF_t} = W_t.\P_{\mid \calF_t}.
\]
We also set 
$$
  \hat{a} \coloneqq a + \theta \sigma^2 +\int_{\calP}\left( \sum_{k\geq 1} x_k\e^{\theta x_k} \ind{|x_k|<1} -x_1\ind{|x_1|<1}\right)\Lambda(\dd \x),
$$
where \eqref{eqn:lambdaLevyProcess} and \eqref{eqn:exponentialIntegrability} ensure that the integral above is well-defined and finite. 
 
Then let $\hat \calN (\dd t, \dd \x)$ be a Poisson point process on $[0,\infty)\times \calP$ with intensity $\dd t \otimes \hat \Lambda(\dd \x)$, and recall that $(B_t)_{t\geq 0}$ denotes an independent Brownian motion. For each atom of $\hat \calN$, say $(t,\x)$, we sample independently of the other atoms an index $n\geq 1$ with probability proportional to $\e^{\theta x_n}$ and denote it by $*$, omitting the dependence in $(t,\x)$ in the notation  for the sake of simplicity. In particular $\P(*=n\mid \hat \calN)=\e^{\theta x_n} /\crochet {\x, \e_{\theta}}$.
Next note, again thanks to  \eqref{eqn:lambdaLevyProcess} and \eqref{eqn:exponentialIntegrability}, that 
$$
  \int_\mathcal{P} \sum_{n\geq 1}\e^{\theta x_n}(1 \wedge x_n^2) \Lambda(\dd \x) < \infty.
$$
This enables us to define the (compensated) Poissonian integrals below and set
$$
  \hat \xi_t\coloneqq \sigma B_t + \hat a t+ \int_{[0,t]\times \calP} x_*\ind{|x_*|<1} \hat\calN^{\rm(c)}(\dd s,\dd \x) + \int_{[0,t]\times \calP} x_*\ind{|x_*|\geq 1} \hat\calN(\dd s,\dd \x)
$$
for $t\geq 0$. Plainly, $\hat \xi$ is another L\'evy process, which is referred to as the spine. 

\begin{lemma} \label{L1}
The characteristic exponent of $\hat \xi$ is given by 
$$\hat\Phi(r)\coloneqq \kappa(\theta+ir)-\kappa(\theta), \qquad r\in\R,$$
and it holds that
$$\lim_{t\to \infty} t^{-1} \hat \xi_t= \kappa'(\theta)\qquad\text{a.s.}$$
\end{lemma}

\begin{proof} By Poissonian calculus, we get $\E\left(\exp(ir\hat \xi_t)\right) = \exp (t\hat \Phi(r))$ with 
$$
\hat\Phi(r)= -\frac{\sigma^2}{2} r^2 + i \hat a r+\int_{\calP}\sum_{n\geq1}\e^{\theta x_n}\left(\e^{irx_n}-1-irx_n\ind{|x_n|<1}\right) \Lambda(\dd \x)$$
and the first claim follows readily by substitution. Further, the random variable $\hat \xi_1$ is integrable with expectation
$$\hat a + \int_{\calP}\sum_{n=1}^{\infty}x_n \e^{\theta x_n}\ind{|x_n|\geq 1}\Lambda(\dd \x).$$
Again after substitution, we find $\E(\hat \xi_1)=\kappa'(\theta)$, and we conclude applying the law of large numbers for L\'evy processes that $\hat \xi_t\sim \kappa'(\theta)t$  as $t\to \infty$, a.s.
\end{proof}

We can now provide a description of the spinal decomposition for the branching L\'evy process, which is tailored for our purpose. In this direction, we construct a particle system much in the same way as we did for branching L\'evy processes, except that we use the Poisson point process $\hat \calN$ instead of $\calN$, and the trajectory $\hat{\xi}$ to define the so-called Eve particle and its offspring. Specifically, for each atom, say $(t,\x)$, of $\hat \calN$, we view $t$ as the time when the spine jumps to position $\hat{\xi}_{t-} + x_*$, while giving birth to a sequence of children located at $\hat\xi_{t-}+x_j$ for all $j \neq *$. Each of the newborn particles immediately starts an independent copy of the original branching Lévy process $Z$ from its current position. Writing $\hat Z_t$ for the random point measure whose atoms are given by the positions of the particles in the system at time $t$, we are now able to state a simple version of the spine decomposition, and refer to Theorem 5.2 of  Shi and Watson  \cite{ShW17+} for a more detailed version in the setting of compensated fragmentations.
\begin{lemma}
  \label{eqn:spinal}
 The process $\hat Z=(\hat Z_t)_{t\geq 0}$ above has the same law as  $Z$ under $\hat\P$.
\end{lemma}
For the reader's convenience, we sketch a proof of this statement.
\begin{proof}
We assume in a first time that $Z$ has a finite birth intensity, in the sense that
\begin{equation}
  \int_\mathcal{P} \sum_{n\geq 2} \ind{x_n > -\infty} \Lambda(\dd \x) < \infty.
\end{equation}
In this case, the branching Lévy process is of the type considered by Kyprianou \cite{Kyp1999}, it can be viewed as a classical Uchiyama-type branching random walk to which independent spatial displacements are superposed. Specifically, each particle moves according to an independent Lévy process until an exponential time of parameter $\Lambda(x_1 = -\infty \text{ or }x_2 > -\infty)$ at which a death or reproduction event occurs. Lemma~\ref{eqn:spinal} is then a simple instance of the spinal decomposition for branching Markov processes, that can be found in \cite{HaH09} (see also \cite[Section 3]{Mai16} for an overview of similar results).

To treat the general case, we use the observation made in \cite[Section 5]{BeM17+} that any branching Lévy process can be constructed as the increasing limit of branching Lévy processes with finite birth intensity. Specifically,  for any $n \in \N$ and $\x \in \mathcal{P}$, we set
\[
  \pi_n(\x) = (x_j - \infty \ind{x_j<-n}, j \in \N),
\]
that is, $\pi_n(\x)$ is obtained from $\x$ by deleting every particle located in $(-\infty, -n)$.
We denote by $Z^{(n)}$ the branching Lévy process obtained from $Z$ using the image of the point measure $\mathcal{N}$ by $(t,\x) \mapsto (t,\pi_n(\x))$. In words, $Z^{(n)}$ is obtained from $Z$ by killing each particle (of course together with its own descent) at the time it makes a jump smaller than $-n$. We write $\kappa^{(n)}$ for the cumulant generating function of $Z^{(n)}$ and $W^{(n)}$
for the additive martingale 
$$W^{(n)}_t=\exp(-t\kappa^{(n)}(\theta))\crochet{Z^{(n)}_t, \e_{\theta}}.$$

We construct $\hat{Z}^{(n)}$ in a similar way, that is by killing every particle in $\hat Z$ at the time it makes a jump smaller than $-n$. Beware that $\hat{Z}^{(n)}$ is different from the point measure valued process $\hat{Z^{(n)}}$ which is associated the branching L\'evy process $Z^{(n)}$, 
as described earlier in this section. Nevertheless, there is a simple connection between the two: if we write
\[
  T^{(n)}_* \coloneqq \inf\{t > 0 : \hat{\xi}_t - \hat{\xi}_{t-} < -n \},
\]
for the time at which the spine particle of $\hat{Z}$ is killed in $\hat{Z}^{(n)}$, then for every $t\geq 0$, the processes
$(\hat{Z^{(n)}_s}: 0\leq s \leq t)$ and $(\hat{Z}^{(n)}_s: 0\leq s \leq t)$ have the same law
conditionally on $T^{(n)}_*>t$.

Indeed, observe that the waiting time $T^{(n)}_*$ can be rewritten
\[
  T^{(n)}_* = \inf\{ t > 0: (t,\x) \text{ atom of } \hat{\calN} \text{ with } x_* < -n \},
\]
hence, conditionally on $T^{(n)}_* > t$, $\hat{\calN}$ is a Poisson point process conditioned on the fact that each atom $(s,\x)$ with $s < t$ satisfies $x_* \geq -n$. By classical Poissonian properties, the image measure of this process by $(\mathrm{Id},\pi_n)$ is a Poisson point process with intensity $\dd t \hat{\Lambda^{(n)}}(\dd\x)$, where $\Lambda^{(n)}$ is the image measure of $\Lambda$ by~$\pi_n$. Moreover, note that for each atom $(s,\x^{(n)})$ of that censored Poisson point process, the mark is sampled at random, and we have $* =j$ with probability $\e^{\theta x^{(n)}_j}/\crochet{\x^{(n)},\e_\theta}$.

The branching Lévy process $Z^{(n)}$ has  finite birth intensity, and we now see from its spinal decomposition that the law of $\hat{Z}^{(n)}$ on $\calF_t$ conditionally on $T^{(n)}_*>t$, is the same as $W^{(n)}_t.\P_{\mid \calF_t}$. Since $\lim_{n\to \infty}T^{(n)}_*=\infty$ a.s., and (by monotone convergence) $\lim_{n\to \infty}W^{(n)}_t=W_t$ in $L^1(\P)$, we easily conclude that the spinal decomposition also holds for $Z$.
\end{proof}

By a classical observation (see Exercice 3.6 in \cite[p. 210]{Durrett}), the proof of Theorem~\ref{T1} amounts to establishing that $\hat{\P}$-a.s.,
$\limsup_{t\to \infty} W_t<\infty$ if the conditions \eqref{cond1} and \eqref{cond2} hold, and $\limsup_{t\to \infty} W_t=\infty$ otherwise. As a consequence of Lemma~\ref{eqn:spinal}, if we write
$$\hat W_t \coloneqq \e^{-t\kappa(\theta)} \crochet {\hat Z_t, \e_{\theta}},$$
then the process $\hat W$ has the same law as $W$ under $\hat \P$, so the next statement entails the second part of Theorem~\ref{T1}.

\begin{lemma}
\label{L2}
If \eqref{cond2} fails, then $\limsup_{t\to\infty} \hat W_t=\infty$ a.s.
\end{lemma} 

\begin{proof}
From the construction of $\hat Z$, we observe for every atom $(t,\x)$ of $\hat{\calN}$, by focusing on the spine and its children which are born at time $t$, that there is the bound
$$ \hat W_t  \geq \exp(\theta \hat \xi_{t-} -t\kappa(\theta)) \crochet{\x,\e_{\theta}}.$$
Fix $c>0$ with $-c<\theta \kappa'(\theta)-\kappa(\theta)$, and recall from Lemma~\ref{L3} that the failure of \eqref{cond2} entails that
$$\int_{s}^{\infty} \hat \Lambda(\crochet{\x,\e_{\theta}}>\e^{ct})\dd t=\infty \qquad \text{for every }s>0.$$
This implies that the set of times $t\geq 0$ such that the Poisson point process $\hat \calN$ has an atom $(t,\x)$ with $\crochet{\x,\e_{\theta}}>\e^{ct}$ is unbounded a.s., and an appeal to Lemma~\ref{L1} completes the proof. 
\end{proof}

Since we already know from Biggins' theorem that $W_{\infty}=0$ a.s. when \eqref{cond1} fails, we may now turn our attention to the situation where \eqref{cond1} and \eqref{cond2} both hold, and recall that our goal is then to prove that $\limsup_{t\to\infty} \hat W_t<\infty$ a.s. In this direction, we first  write
\begin{equation}\label{eqn:decWhat}
\hat W_t= \exp(\theta \hat\xi_t-t\kappa(\theta)) + (\hat W_t-\exp(\theta \hat\xi_t-t\kappa(\theta)) ).
\end{equation}
Thanks to Lemma~\ref{L1} and \eqref{cond1}, we know that
$$\lim_{t\to \infty} \exp(\theta \hat\xi_t-t\kappa(\theta))=0 \qquad \text{a.s.}$$
We then write $\hat \sigma$ for the sigma-field generated by the Poisson point process $\hat \calN$ and the random indices $*$ which are selected for each of its atoms. Viewing the second term in the right-hand side of \eqref{eqn:decWhat} as the contribution of the descendants of the children of the spine which were born before time $t$,
we get from the spinal decomposition and the martingale property of $W$ for the branching L\'evy process, that there is the identity
\begin{eqnarray}\label{eq:w*}
  W^*_t 
  & \coloneqq &\E\left( \hat W_t-\exp(\theta \hat\xi_t-t\kappa(\theta))  \middle| \hat \sigma\right) \nonumber \\
  & = &\int_{[0,t]\times \calP}\sum_{k\neq *}\exp(\theta (\hat\xi_{s-}+x_k)-s\kappa(\theta)) \hat{\calN}(\dd s, \dd \x) .
\end{eqnarray}
By the conditional Fatou lemma, it now suffices to verify that the process $W^*$ remains bounded a.s. The lemma below thus completes the proof of Theorem~\ref{T1}.

\begin{lemma}
\label{L4}
If \eqref{cond1} and \eqref{cond2} both hold, then $\sup_{t\geq 0}  W^*_t <\infty$ a.s. 
\end{lemma}

\begin{proof}
The process $W^*$ has non-decreasing paths, so we have to check that $W^*_{\infty}<\infty$ a.s. 

Thanks to \eqref{cond2}, we pick $c>0$ sufficiently small so that $\theta\kappa'(\theta)-\kappa(\theta)<-c$, and then, thanks to Lemma~\ref{L1}, we know that the probability of the event 
$$\Omega_b\coloneqq \{\exp(\theta \hat\xi_{s-}-s\kappa(\theta))\leq b\e^{-cs} \text{ for all $s\geq 1$}\}$$
converges to $1$ as $b\to \infty$. Therefore, we conclude that
\[
  \sup_{s \geq 0} \frac{\exp(\theta \hat\xi_{s-}-s\kappa(\theta))}{\e^{-cs}} < \infty, \quad \text{a.s.}
\]
Hence we only need to check the finiteness of the Poissonian integral 
$$ \int_{[0,\infty)\times \calP} \e^{-cs}  \sum_{k\neq *}\e^{\theta x_k} \hat{\calN}(\dd s, \dd \x).$$
 
In this direction, fix $0<c'<c$. Since $\hat{\calN}$ is a Poisson point process with intensity $\dd s \otimes \hat\Lambda(\dd \x)$, it follows from Lemma~\ref{L3} that the set of times $s\geq 0$ such $\hat{\calN}$ has an atom $(s,\x)$ with $\crochet {\x,\e_{\theta}}>\e^{c's}+1$ is finite a.s., and {\it a fortiori} 
$$ \int_{[0,\infty)\times \calP} \e^{-cs}  \sum_{k\neq *}\e^{\theta x_k} \ind{ \crochet {\x,\e_{\theta}}>\e^{c's}+1} \hat{\calN}(\dd s, \dd \x)<\infty \quad \text{a.s.}$$
 
On the other hand, again by Poissonian calculus, 
\begin{align*}
   &\E\left( \int_{[0,\infty)\times \calP} \e^{-cs}  \sum_{k\neq *}\e^{\theta x_k} \ind{ \crochet {\x,\e_{\theta}}\leq \e^{c's}+1} \hat{\calN}(\dd s, \dd \x)\right)\\
  =&  \int_0^{\infty} \dd s \, \e^{-cs} \int_{\calP} \Lambda(\dd \x) \sum_{j\geq 1} \e^{\theta x_j} \sum_{k\neq j} \e^{\theta x_k}\ind{ \crochet {\x,\e_{\theta}}\leq \e^{c's}+1} \\
  \leq & \int_0^{\infty} \dd s \, \e^{-cs} \int_{\calP} \Lambda(\dd \x) \ind{ \crochet {\x,\e_{\theta}}\leq \e^{c's}+1}
\left( \e^{\theta x_1} \sum_{k\geq 2} \e^{\theta x_k} +\sum_{j\geq 2} \e^{\theta x_j} \crochet {\x,\e_{\theta}} \right) \\
  \leq & \int_0^{\infty} \dd s \, \e^{-cs} \int_{\calP} \Lambda(\dd \x)  2 (\e^{c's}+1) \sum_{k\geq 2} \e^{\theta x_k}, 
\end{align*}
where for the first equality, we used  that  the conditional probability given $\x$ that $* =j$ equals  $\e^{\theta x_j}/\crochet{\x,\e_\theta}$, 
and that the Poisson random measure $\hat {\mathcal N}(\dd s, \dd \x)$
 has intensity $\crochet{\x,\e_{\theta}} \dd s \Lambda(\dd x)$. 
By \eqref{eqn:exponentialIntegrability}, the right-hand side is finite, which completes the proof.
\end{proof}

Finally, we turn our attention to the proof of Proposition~\ref{P1}.

\begin{proof}[Proof of Proposition~\ref{P1}] 
Thanks to Theorem 1 of Biggins \cite{biggins1992}, it is enough to check that, under the assumptions of the statement, one has  $\E(W_1^p)<\infty$, or equivalently, that
$$\hat \E(W_1^{p-1})=\E(\hat W_1^{p-1})<\infty.$$ In this direction, we use the decomposition
\eqref{eqn:decWhat} and note first, using Lemma~\ref{L1}, that
\begin{equation}\label{almostdone}
\E\left(\exp((p-1)(\theta \hat \xi_1-\kappa(\theta))\right) = \exp\left(\kappa(p\theta)-p\kappa(\theta)\right)<1.
\end{equation}

Recall  that   $W^*_t$ denotes  the conditional expectation  of
the second term of the sum in the right-hand side of \eqref{eqn:decWhat}  given the sigma-field generated by the Poisson point process $\hat \calN$ and the random indices $*$ which are selected for each of its atoms.
Since $0<p-1<1$, thanks to the conditional version of Jensen's inequality, it suffices to check that $\E((W^*_1)^{p-1})<\infty$. 

In this direction, we use \eqref{eq:w*} and further
distinguish  the atoms $(s,\x)$ of $\hat{\calN}$ depending on whether $\crochet{\x,\e_{\theta}}\leq 2$ or not, and write 
\begin{equation}\label{eq:wstar}
W^*_1\leq AB+C
\end{equation} 
where 
\begin{eqnarray*} A&=& \sup \{\exp((\theta \hat\xi_{s-}-\kappa(\theta)s)) : 0\leq s \leq 1  \},\\
B&=&
 \int_{[0,1]\times \{\crochet{\x,\e_{\theta}}\leq 2\}}   \sum_{i\neq *}\e^{\theta x_i}  \hat{\calN}(\dd s, \dd \x) ,\\
C&=& \int_{[0,1]\times  \{\crochet{\x,\e_{\theta}}> 2\}}  \exp(\theta \hat\xi_{s-}-\kappa(\theta)s) \sum_{i\neq *}\e^{\theta x_i} \hat{\calN}(\dd s, \dd \x).
\end{eqnarray*}

First, it follows from Lemma~\ref{L1}  that the process 
 $$M_s=\exp((p-1)\theta \hat\xi_{s}-(\kappa(p\theta)- \kappa(\theta))s),\qquad s\geq 0$$ is a martingale.
From our assumption  $\kappa(q\theta)<\infty$ for some $q>p$, we further see that  
 $$\E(M_1^{(q-1)/(p-1)})<\infty,$$
and then, from Doob's inequality, that 
$$\E\left( \sup_{0\leq s \leq 1} \exp((q-1)\theta \hat\xi_{s})\right)<\infty.$$ 
This proves that 
\begin{equation} \label{eq:Aq-1}\E(A^{q-1})<\infty.
\end{equation}

We next check that $B$ has a finite exponential moment. 
Observe from a combination of the formula for the Laplace transform of Poissonian integrals and Campbell's formula (see, e.g. Sections 2.2 and 3.3 in \cite{LastPen}), that for every Poisson random measure $N$ and every nonnegative function $f$, there is the identity
$$\E\left(\exp\left(\int f(y)N(\dd y)\right)\right) = \exp\left(\E\left(\int (\e^{f(y)}-1)N(\dd y)\right)\right).$$
This gives
\begin{eqnarray*} 
\log  \E\left( \exp(B)\right) 
&= &\E\left(  \int_{[0,1]\times \{\crochet{\x,\e_{\theta}}\leq 2\}} \left( \exp\left ( \sum_{i\neq *}\e^{\theta x_i}\right ) -1\right ) \hat{\calN}(\dd s, \dd \x)
\right)\\
&\leq  &\e^2\E\left(  \int_{[0,1]\times \{ \crochet{\x,\e_{\theta}}\leq 2\}}  \sum_{i\neq *}\e^{\theta x_i}\hat{\calN}(\dd s, \dd \x)
\right).
\end{eqnarray*}
Since $\hat{\calN}$ is a Poisson random measure with intensity $\dd s \times \crochet{\x,\e_{\theta}}\Lambda(\dd \x)$,
another application of Campbell's formula enables us to express the last quantity 
in the form
\begin{eqnarray*}  & & \e^2 \int_{\{ \crochet{\x,\e_{\theta}}\leq 2\}}  \sum_{k\geq 1}   \e^{\theta x_k} \sum_{j\neq k}   \e^{\theta x_j} \Lambda(\dd \x)\\
&\leq & \e^2 \int_{\{ \crochet{\x,\e_{\theta}}\leq 2\}}\left(   \e^{\theta x_1} \sum_{j\geq 2}   \e^{\theta x_j} 
+ \sum_{k\geq 2}   \e^{\theta x_k} \crochet{\x,\e_{\theta}}\right) \Lambda(\dd \x)\\
&\leq &4 \e^2\int_{\calP}   \sum_{j\geq 2}   \e^{\theta x_j} \Lambda(\dd \x).
\end{eqnarray*}
By \eqref{eqn:exponentialIntegrability} the last quantity is finite. This entails $\E(\exp(B))<\infty$, and 
{\it a fortiori} that $\E(B^{(p-1)(q-1)/(q-p)})<\infty$. We conclude by H\"older's inequality from \eqref{eq:Aq-1} that 
\begin{equation} \label{eq:ABq-1}\E((AB)^{p-1})<\infty.
\end{equation}

Finally, we turn our attention to $C$. 
Since $0<p-1\leq 1$ and $\hat{\calN}(\dd s, \dd \x)$ is a (random) point measure, for every nonnegative process $(H_s)_{s\geq 0}$, the inequality 
$$\left( \int_{[0,1]\times \calP} H_s \hat{\calN}(\dd s, \dd \x)\right)^{p-1}\leq   \int_{[0,1]\times \calP} H_s^{p-1} \hat{\calN}(\dd s, \dd \x)$$
holds, as $\Vert \mathbf{y} \Vert_{1/(p-1)} \leq \Vert \mathbf{y} \Vert_{1}$ for all real-valued sequences $\mathbf{y}$. Hence, there is the inequality
$$
C^{p-1} \leq  \int_{[0,1]\times  \{\crochet{\x,\e_{\theta}}> 2\}}  \exp((p-1)(\theta \hat\xi_{s-}-\kappa(\theta)s)) \left( \sum_{i\neq *}\e^{\theta x_i} \right)^{p-1}\hat{\calN}(\dd s, \dd \x).
$$
The left-continuous process $s\mapsto \exp((p-1)(\theta \hat\xi_{s-}-\kappa(\theta)s))$ is predictable; recall further that
 the conditional probability given $\x$ that $* =k$ equals  $\e^{\theta x_k}/\crochet{\x,\e_\theta}$, 
and that the Poisson point measure $\hat{\calN}(\dd s, \dd \x)$ has intensity $\crochet{\x,\e_{\theta}}\dd s \Lambda(\dd \x)$. We now see that $\E\left( C^{p-1}\right)$ can be bounded from above by
$$
\int_{\{{\crochet{\x,\e_{\theta}}>2}\}} \sum_{k\geq 1} \e^{\theta x_k}\left(\sum_{i\neq k}\e^{\theta x_i}\right)^{p-1} \Lambda(\dd \x) \times \E\left( \int_0^1\e^{(p-1)(\theta \hat\xi_{s-}-\kappa(\theta)s))} \dd s \right).  $$
Finally, recall from \eqref{almostdone} that 
$$\E(\e^{(p-1)(\theta \hat\xi_{s-}-\kappa(\theta)s))}) =\E(\e^{(p-1)(\theta \hat\xi_{s}-\kappa(\theta)s))}) \leq 1 \quad \text{for all }s\geq 0,$$
and therefore
$$\E\left( C^{p-1}\right) \leq \int_{\{{\crochet{\x,\e_{\theta}}>2}\}} {\crochet{\x,\e_{\theta}}}^p \Lambda(\dd \x).
$$
We conclude from \eqref{cond3} that $\E(C^{p-1})<\infty$, and hence, from \eqref{eq:wstar} and \eqref{eq:ABq-1}, that 
$\E((W^*_1)^{p-1})<\infty$. This completes the proof.
\end{proof}

\end{document}